\documentclass[a4paper,11pt]{article}
\usepackage{amsmath,amsfonts,amsthm}
\usepackage[a4paper,centering,hscale=0.7,vscale=0.7]{geometry}

\newtheorem{prop}{Proposition}%[section]
\newtheorem{thm}[prop]{Theorem}
\newtheorem{lemma}[prop]{Lemma}

\newtheorem{defi}[prop]{Definition}
\theoremstyle{remark}
\newtheorem{rmk}[prop]{Remark}
\theoremstyle{definition}

\renewcommand{\P}{\mathbb{P}}
\newcommand{\E}{\mathbb{E}}
\renewcommand{\H}{\mathbb{H}}
\renewcommand{\L}{\mathbb{L}}
\newcommand{\LL}{\mathcal{L}}
\newcommand{\erre}{\mathbb{R}}

\newcommand{\m}{\bar{\mu}}
\newcommand{\ep}[1]{{#1}^\varepsilon}
\newcommand{\lip}{\dot{C}^{0,1}}
\newcommand{\ip}[2]{\langle#1,#2\rangle}
\newcommand{\bip}[2]{\big\langle#1,#2\big\rangle}

\newsymbol\lesssim 132E

\title{Well-posedness for a class of dissipative stochastic
  evolution equations with Wiener and Poisson noise}

\author{Carlo Marinelli\thanks{Facolt\`a di Economia, Libera
    Universit\`a di Bolzano, I-39100 Bolzano, Italy.}}

\date{\normalsize 12 October 2011}

\begin{document}
\maketitle

\begin{abstract}
  We prove existence and uniqueness of mild and generalized solutions
  for a class of stochastic semilinear evolution equations driven by
  additive Wiener and Poisson noise. The non-linear drift term is
  supposed to be the evaluation operator associated to a continuous
  monotone function satisfying a polynomial growth condition. The
  results are extensions to the jump-diffusion case of the
  corresponding ones proved in \cite{cm:EJP10} for equations driven by
  purely discontinuous noise.
%
%\medskip\par\noindent
%
%\emph{2010 Mathematics Subject Classification:} 60H15, 60G51, 60G57.
%\smallskip\par\noindent
%
%\emph{Keywords and phrases:} stochastic PDE with jumps, 
\end{abstract}

\section{Introduction}
The purpose of this note is to show that stochastic evolution equations
of the type
\begin{equation}
  \label{eq:musso}
  du(t) + Au(t)\,dt + f(u(t))\,dt = B(t)\,dW(t)
  + \int_Z G(z,t)\,\m(dz,dt),
  \qquad u(0)=u_0,
\end{equation}
where $A$ is a linear $m$-accretive operator on a Hilbert space $H$,
$f:\erre\to\erre$ a monotone increasing function of polynomial growth,
$W$ is a cylindrical Wiener noise on $H$, and $\m$ is a compensated
Poisson random measure, admit a unique mild solution. Precise
assumptions on the space on which the equation is considered and on
the data of the problem are given in the next section.

Global well-posedness of (\ref{eq:musso}) in the case of purely
discontinuous noise (i.e. with $B \equiv 0$) has been proved in
\cite{cm:EJP10} showing that solutions to regularized equations
converge to a process which solves the original equation. This is
achieved proving a priori estimates for the approximating processes by
rewriting the regularized stochastic equations as deterministic
evolution equations with random coefficients and using monotonicity
arguments. These a priori estimates essentially rely, in turn, on a
maximal inequality of Bichteler-Jacod type for stochastic convolutions
on $L_p$ spaces with respect to compensated Poisson random measures,
also proved in \cite{cm:EJP10}.

The well-posedness results of \cite{cm:EJP10} will be here extended to
the more general class of equation (\ref{eq:musso}). We shall adapt
the method used in \cite{cm:EJP10}, but instead of rewriting the
regularized (stochastic) equations as deterministic equations with
random coefficients, we shall rewrite them as stochastic equations
driven just by Wiener noise (we might say that, in a sense, we ``hide
the jumps''), the solutions of which will be shown to satisfy suitable
a priori estimates allowing to pass to the limit in the regularized
equations.

The result might be interesting even in the case of equations driven
only by a Wiener process (i.e. with $G \equiv 0$). In fact, the usual
approach to establish well-posedness for such equations
(cf. e.g. \cite{cerrai-libro,DZ92}) is to rewrite them as
deterministic equations with random coefficients and to consider them
on a Banach space of continuous functions. This approach requires the
stochastic convolution to have paths in such a space of continuous
functions. The latter condition is not needed in our setting.

\medskip

Let us conclude this introductory section with some words about
notation used throughout the paper: $a \lesssim b$ stands for $a \leq
Nb$ for some constant $N$ (if the constant $N$ depends on parameters
$p_1,\ldots,p_n$ we shall write $N(p_1,\ldots,p_n)$ and $a
\lesssim_{p_1,\ldots,p_n}$, respectively). For any $p \geq 0$, we set
$p^*:=p^2/2$.
Given two (metric) spaces $E$, $F$, we shall denote the space of
Lipschitz continuous functions from $E$ to $F$ by $\lip(E \to F)$.
The duality mapping of a Banach space $E$ with dual $E'$ and duality
form $\ip{\cdot}{\cdot}$ is the (multivalued) map $J:E \to 2^{E'}$,
$J: x \mapsto \{x^* \in E': \ip{x^*}{x}=\|x\|_E^2\}$.

\section{Main result}
Let $(\Omega,\mathcal{F},(\mathcal{F}_t)_{0 \leq t \leq T},\P)$, with
$T>0$ fixed, be a filtered probability space satisfying the usual
conditions, and let $\E$ denote expectation with respect to $\P$. All
stochastic elements will be defined on this stochastic basis, and any
equality or inequality between random quantities will be meant to hold
$\P$-almost surely.
Let $(Z,\mathcal{Z},m)$ be a measure space, $\m$ a Poisson measure on
$[0,T] \times Z$ with compensator $\mathrm{Leb} \otimes m$, where Leb
stands for Lebesgue measure. 
Let $D$ be an open bounded subset of $\erre^d$ with smooth boundary.
All Lebesgue spaces on $D$ will be denote without explicit mention of
the domain, e.g. $L_p:=L_p(D)$. We shall sometimes denote $L_2$ by
$H$.
Given $q \geq 1$ and a Banach space $E$, we shall denote the set of all
$E$-valued random variables $\xi$ such that $\E|\xi|^q<\infty$ by
$\mathbb{L}_q(E)$. We call $\H_p(E)$ the set of all adapted $E$-valued
processes such that
\[
\| u \|_{\H_p(E)} := 
\big( \E\sup_{t \leq T} \| u(t) \|_E^p \big)^{1/p} < \infty.
\]
For compactness of notation, we shall also write $\mathbb{L}_q$ in
place of $\mathbb{L}_q(L_q)$. and $\H_q$ in place of $\H_q(L_q)$. We
shall denote by $W$ a cylindrical Wiener process on $L_2(D)$.

Let $f:\erre \to \erre$ be a monotone increasing function with
$f(0)=0$, for which there exists $p \geq 2$ such that $|f(r)| \lesssim
1+|r|^{p/2}$ for all $r \in \erre$.

Let $A$ be a linear (unbounded) $m$-accretive operator in the spaces
$H=L_2$, $E:=L_p$ and $L_{p^*}$, and assume that $S$, the strongly
continuous semigroup generated by $-A$ on $E$, is analytic. We shall
not distinguish notationally the different realizations of $A$ and $S$
on the above spaces.

Denoting by $\gamma(H \to E)$ the space of $\gamma$-radonifying
operators from $H$ to $E$, for any $q \geq 1$, the class of adapted
processes $B: [0,T] \to \mathcal{L}(H \to E)$ such that
\[
\|B\|^q_{\LL^\gamma_q} := \E \int_0^T \|B(t)\|^q_{\gamma(H \to E)}\,dt < \infty
\]
will be denoted by $\LL^\gamma_q(H \to E)$. Similarly, denoting the
predictable $\sigma$-algebra by $\mathcal{P}$ and the Borel
$\sigma$-algebra of $\erre^d$ by $\mathcal{B}(\erre^d)$, the space of
$\mathcal{P} \otimes \mathcal{Z} \otimes
\mathcal{B}(\erre^d)$-measurable processes $g:[0,T] \times Z \times D
\to \erre$ such that
\[
\|g\|_{\LL^m_q}^q := 
\E\int_0^T\!\!\int_Z \|g(t,z)\|^q_{L_q}\,m(dz)\,dt +
\E\int_0^T\!\!\Big(\int_Z \|g(t,z)\|_{L_q}^2\,m(dz)\Big)^{q/2}\,dt < \infty
\]
will be denoted by $\LL^m_q$. It was proved in \cite{cm:EJP10} that,
for any strongly continuous semigroup of positive contractions $R$ on
$L_q$, $q \in [2,\infty[$, one has the maximal inequality
\begin{equation}     \label{eq:BJ}
\E \sup_{t \leq T} \Big\| 
\int_0^t\!\!\int_Z R(t-s) g(s,z)\,\m(ds,dz) \Big\|^q_{L_q}
\lesssim \|g\|_{\LL^m_q}^q.
\end{equation}

Let us now define mild and generalized solutions of (\ref{eq:musso}).
\begin{defi}
  Let $u_0$ be an $H$-valued $\mathcal{F}_0$-measurable random
  variable. A (strongly) measurable adapted $H$-valued process $u$ is
  a mild solution to (\ref{eq:musso}) if, for all $t \in [0,T]$,
  \begin{multline*}
  u(t) + \int_0^t S(t-s)f(u(s))\,ds \\= S(t)u_0 
  + \int_0^t S(t-s)B(s)\,dW(s)
  + \int_0^t\!\!\int_Z S(t-s)G(s,z)\,\m(dz,ds)
  \end{multline*}
  and all integrals are well-defined.
\end{defi}
As is well known, the stochastic convolution with respect to $W$ is
well-defined if the operator $Q_t$ is nuclear for all $t \in [0,T]$,
where
\[
Q_t := \int_0^t S(t-s)B(s)B^*(s)S^*(t-s)\,ds.
\]
This condition is verified, for instance, if $B \in \LL^\gamma_2$, i.e. if
\[
\E \int_0^T \|B(s)\|^2_{\gamma(H \to H)}\,ds < \infty
\]
(recall that $\gamma(H \to H)$ is just the space of Hilbert-Schmidt
operators from $H$ to itself).
Similarly, the stochastic convolution with respect to $\m$ is
well-defined if $G \in \LL^m_2$, i.e. if
\[
\E\int_0^T\!\!\int_Z \|G(s,z)\|^2_{L_2}\,m(dz)\,ds < \infty.
\]
The deterministic convolution term is well-defined if $f(u) \in
L_1([0,T] \to H)$, or if $u \in L_{p/2}([0,T] \to L_p)$.

\begin{defi}
  A process $u \in \mathbb{H}_2(T)$ is a generalized solution to
  (\ref{eq:musso}) if there exist $\{u_{0n}\}_n \subset \L_p$,
  $\{B_n\}_n \subset \LL^\gamma_p$, $\{G_n\}_n \subset
  \mathcal{L}^m_{p^*}$, and $\{u_n\}_n \subset \mathbb{H}_2(T)$ such
  that $u_{0n} \to u_0$ in $\L_2$, $B_n \to B$ in $\LL^\gamma_2$, $G_n
  \to G$ in $\LL^m_2$ and $u_n \to u$ in $\mathbb{H}_2(T)$ as $n \to
  \infty$, where $u_n$ is the (unique) mild solution of
  \[
  du_n(t) + Au_n(t)\,dt + f(u_n(t))\,dt = B_n(s)\,dW(t)
  + \int_Z G_n(z)\,\m(dt,dz),
  \qquad u_n(0)=u_{0n}.
  \]
\end{defi}

Here are the results, which will be proved in the next sections.
\begin{thm}     \label{thm:main}
  Assume that $u_0 \in \L_p$, $B \in \LL^\gamma_p$ and $G \in
  \LL^m_{p^*}$. Then there exists a unique c\`adl\`ag mild solution $u
  \in \H_2$ to equation (\ref{eq:musso}) such that $f(u) \in L^1[0,T]
  \to H)$.
\end{thm}
\begin{thm}     \label{thm:gen}
  Assume that $u_0 \in \L_2$, $B \in \LL^\gamma_2$, $G \in
  \LL^m_2$. Then there exists a unique generalized solution to
  equation (\ref{eq:musso}).
\end{thm}

\begin{rmk}
  By inspection of the corresponding proof in \cite{cm:EJP10}, it is
  clear that the same argument applies to Theorem \ref{thm:main} if
  one assumes $B \in \LL^\gamma_{p^*}$, i.e.
  \[
  \E \sup_{t \leq T} \| W_A(t) \|_{L_{p^*}}^{p^*} < \infty.
  \]
  In the proof of Theorem \ref{thm:main} below we show that $B \in
  \LL^\gamma_{p^*}$ is too strong an assumption, and that $B \in
  \LL^\gamma_p$ is enough. It is natural to conjecture that also $G
  \in \LL^m_{p^*}$ is too strong, and it should suffice to assume $G
  \in \LL^m_p$. Unfortunately, thus far we have not been able to
  replace the exponent $p^*$ by $p$ in the hypotheses on $G$ of
  Theorem \ref{thm:main}.
\end{rmk}

\section{Proofs}
\subsection{Proof of Theorem \ref{thm:main}}
Let $f_\lambda:=\lambda^{-1}(I-(I+\lambda f)^{-1})$, $\lambda>0$, be
the Yosida approximation of $f$, and recall that $f_\lambda \in
\dot{C}^{0,1}(\erre)$, with $\|f_\lambda\|_{\lip} \leq 2/\lambda$.
Let us consider the regularized equation
\begin{equation}     \label{eq:reg}
du_\lambda(t) + Au_\lambda(t)\,dt + f_\lambda(u_\lambda(t))\,dt = B(t)\,dW(t)
+ \int_Z G(z,t)\,\m(dz,dt),
\qquad u_\lambda(0)=u_0.
\end{equation}
Assuming that $B \in \LL^\gamma_p$ and $G \in \LL^m_p$, one could
prove by a fixed point argument that (\ref{eq:reg}) admits a unique
c\`adl\`ag mild $E$-solution (by which we mean, here and in the
following, a mild solution with values in $E$). However, we prefer to
proceed in a less direct way, for reasons that will become apparent
later. In particular, we ``hide the jumps'' in (\ref{eq:reg}) writing
an equation for the difference between $u_\lambda$ and the stochastic
convolution with respect to the Poisson random measure as follows:
setting, for notational compactness,
\[
W_A(t) := \int_0^t S(t-s)B(s)\,dW(s),
\qquad
G_A(t) := \int_0^t\!\!\int_Z S(t-s)G(s,z)\,\m(ds,dz),
\]
the integral form of (\ref{eq:reg}) reads
\begin{equation}     \label{eq:regm}
u_\lambda(t) + \int_0^t S(t-s)f_\lambda(u_\lambda(s))\,ds
= S(t)u_0 + W_A(t) + G_A(t),
\end{equation}
which can be equivalently written as
\[
u_\lambda(t)-G_A(t) + \int_0^t S(t-s)f_\lambda(u_\lambda(s)-G_A(s)+G_A(s))\,ds
= S(t)u_0 + W_A(t),
\]
hence also, setting $v_\lambda:=u_\lambda-G_A$ and
$\tilde{f}_\lambda(t,y):=f_\lambda(y+G_A(t))$, for $y \in \erre$ and
$t \geq 0$, as
\[
v_\lambda(t) + \int_0^t S(t-s)\tilde{f}_\lambda(v_\lambda(s))\,ds
= S(t)u_0 + W_A(t),
\]
which is the mild form of
\begin{equation}     \label{eq:vl}
dv_\lambda(t) + Av_\lambda(t)\,dt + \tilde{f}_\lambda(t,v_\lambda(t))\,dt
= B(t)\,dW(t), 
\qquad v_\lambda(0)=u_0.
\end{equation}
It is clear that $v_\lambda$ is a mild $E$-solution of (\ref{eq:vl})
if and only if $v_\lambda+G_A$ is a mild $E$-solution of (\ref{eq:reg}).

In the next Proposition we show that (\ref{eq:vl}) admits a unique
mild $E$-solution $v_\lambda$, hence identifying also the unique $E$-mild
solution of (\ref{eq:reg}).
\begin{prop}
  If $u_0 \in \L_p$, $B \in \LL^\gamma_p$ and $G \in \LL^m_p$, then
  equation (\ref{eq:vl}) admits a unique c\`adl\`ag mild $E$-solution
  $v_\lambda \in \H_p$. Therefore equation (\ref{eq:reg}) admits a
  unique c\`adl\`ag mild $E$-solution $u_\lambda \in \H_p$, and
  $u_\lambda=v_\lambda+G_A$.
\end{prop}
\begin{proof}
  We use a fixed point argument on the space $\H_p$. Let us consider
  the operator
  \[
  \mathfrak{F}: \H_p \ni \phi \mapsto \Big(
  t \mapsto S(t)u_0 - \int_0^t S(t-s)\tilde{f}_\lambda(s,\phi(s))\,ds
  + W_A(t) \Big).
  \]
  We shall prove that $\mathfrak{F}$ is a contraction on $\H_p$, if
  $T$ is small enough. Since $u_0 \in \L_p$ and $S$ is strongly
  continuous on $L_p$, it is clear that we can (and will) assume,
  without loss of generality, that $u_0=0$. Then
  \[
  \big\|\mathfrak{F}(\phi)\big\|_{\H_p} \leq 
  \big\| S \ast \tilde{f}_\lambda(\cdot,\phi) \big\|_{\H_p} +
  \big\|W_A\big\|_{\H_p}.
  \]
  By a maximal inequality for stochastic convolutions we have
  \[
  \|W_A\|^p_{\H_p} \lesssim \E\int_0^T \|B(t)\|^p_{\gamma(H \to E)}\,dt,
  \]
  where the right-hand side is finite by assumption. Moreover,
  Jensen's inequality and strong continuity of $S$ on $L_p$ yield
  \[
  \E\sup_{t \leq T} \Big\|
  \int_0^t S(t-s) \tilde{f}_\lambda(s,\phi(s))\,ds \Big\|_E^p
  \lesssim_T \E\sup_{t \leq T} \big\| \tilde{f}_\lambda(t,\phi(t)) \big\|_E^p.
  \]
  Since $\|f_\lambda\|_{\lip} \leq 2/\lambda$, we have
  \[
  | \tilde{f}_\lambda(t,x) - \tilde{f}_\lambda(t,y) | = 
  | f_\lambda(x+G_A(t)) - f_\lambda(y+G_A(t)) | \leq
  \frac{2}{\lambda} |x - y|,
  \]
  hence
  \begin{align*}
    | \tilde{f}_\lambda(t,x) | &\leq | \tilde{f}_\lambda(t,x) -
    \tilde{f}_\lambda(t,0) | + | \tilde{f}_\lambda(t,0) |\\
    &\leq \frac{2}{\lambda} |x| + | f_\lambda(G_A(t)) | \leq
    \frac{2}{\lambda} |x| + \frac{2}{\lambda}|G_A(t)|,
  \end{align*}
  thus also
  \[
  \E\sup_{t \leq T} \big\| \tilde{f}_\lambda(t,\phi(t)) \big\|_E^p
  \lesssim_\lambda \E\sup_{t \leq T} \|\phi(t)\|_E^p 
  + \E\sup_{t \leq T} \|G_A(t)\|_E^p,
  \]
  where the right-hand side is finite because of (\ref{eq:BJ}) and
  because $G \in \LL^m_p$ by hypothesis. We have thus proved that
  $\mathfrak{F}(\H_p) \subseteq \H_p$.
  Since $x \mapsto \tilde{f}_\lambda(t,x,\omega)$ is Lipschitz
  continuous, uniformly over $t \in [0,T]$ and $\omega \in \Omega$,
  analogous computations show that $\mathfrak{F}$ is Lipschitz on
  $\H_p$, with a Lipschitz constant that depends continuously on
  $T$. Choosing $T=T_0$, for a small enough $T_0$ such that
  $\mathfrak{F}$ is a contraction, and then covering the interval
  $[0,T]$ by intervals of lenght $T_0$, one obtains the desired
  existence and uniqueness of a fixed point of $\mathfrak{F}$ in a
  standard way.
\end{proof}

\begin{rmk}
  (i) Note that we have assumed the more natural condition $G \in
  \LL^m_p$ for the well-posedness of the regularized equation
  (\ref{eq:reg}) rather than $G \in \LL^m_{p^*}$. Let us show that the
  latter condition also ensures that $\|G_A\|_{\H_p}$ is finite: since
  $D$ has finite Lebesgue measure and $p^*=p^2/2 \geq p$, H\"older's
  inequality implies
  \[
  \E \sup_{t \leq T} \|G_A(t)\|^p_{L_p} \lesssim_{D}
  \E \sup_{t \leq T} \|G_A(t)\|^p_{L_{p^*}}
  \leq \Big(\E\sup_{t \leq T} \|G_A(t)\|^{p^*}_{L_{p^*}}\Big)^{2/p} < \infty.
  \]

  \noindent
  (ii) The previous existence and uniqueness result also follows by an
  adaptation of \cite[Thm.~6.2]{vNVW}, which is a more general and
  more precise result about well-posedness for equations with Wiener
  noise and Lipschitz coefficients. In \cite{vNVW} the nonlinearity in
  the drift is Lipschitz continuous and satisfies a linear growth
  condition with a constant that does not depend on $t$ and $\omega$,
  hence it does not apply directly to our situation. A reasoning
  completely analogous to the above one permits however to circumvent
  this problem.
\end{rmk}

We shall need the following a priori estimate for the solution to the
regularized equation (\ref{eq:reg}).
\begin{lemma}
  Assume that $u_0 \in \L_p$, $B \in \LL^\gamma_p$ and $G \in
  \LL^m_{p^*}$. Then there exists a constant $N$, independent of
  $\lambda$, such that
  \[
  \E\sup_{t \leq T} \| u_\lambda(t) \|_E^p \leq N\big(1 + \E\|u_0\|_E^p\big).
  \]
\end{lemma}
\begin{proof}
  Let $v_\lambda$ be the mild $E$-solution to (\ref{eq:vl}). For
  $\varepsilon>0$, set
  \begin{align*}
  \ep{u}_0 &:=(I+\varepsilon A)^{-1}u_0,&
  \ep{B}(t) &:= (I+\varepsilon A)^{-1} B(t),\\
  g_\lambda(t) &:= \tilde{f}_\lambda(t,v_\lambda),&
  \ep{g}_\lambda(t) &:= (I+\varepsilon A)^{-1} g_\lambda(t),
  \end{align*}
  and let $\ep{w}_\lambda$ be the mild $E$-solution to
  \[
  d\ep{w}_\lambda + Aw_\lambda\,dt + \ep{g}_\lambda\,dt = \ep{B}\,dW,
  \qquad \ep{w}_\lambda(0)=\ep{u}_0,
  \]
  that is
  \[
  \ep{w}_\lambda(t) = S(t)\ep{u}_0 
  - \int_0^t S(t-s) \ep{g}_\lambda(s)\,ds
  + \int_0^t S(t-s) \ep{B}(s)\,ds
  \]
  for all $t \in [0,T]$. It is easily seen that $\ep{w}_\lambda$ is a
  strong solution, i.e. that one has
  \[
  \ep{w}_\lambda(t) + \int_0^t \big( A\ep{w}_\lambda(s)
  + \ep{g}_\lambda(s)\big)\,ds = \ep{u}_0 + \int_0^t B(s)\,dW(s)
  \]
  for all $t \in [0,T]$, and that $\ep{w}_\lambda=(I+\varepsilon
  A)^{-1}v_\lambda \to v_\lambda$ in $\H_p$ as $\varepsilon \to 0$. We
  are going to apply It\^o's formula (in particular we shall use the
  version in \cite[Thm.~3.1]{vNZ}) to obtain estimates for
  $\|\ep{w}_\lambda\|_E^p$. To this purpose, we have to check that
  \[
  \E \Big( \int_0^T \|b(t)\|_E\,dt \Big)^p < \infty,
  \]
  where $b:=A\ep{w}_\lambda + \ep{g}_\lambda$. One has
  \[
  \E \Big( \int_0^T \|b(t)\|_E\,dt \Big)^p \lesssim
  \E \int_0^T \|b(t)\|_E^p\,dt \lesssim
  \E \int_0^T \|A\ep{w}_\lambda\|_E^p\,dt
  + \E \int_0^T \|\ep{g}_\lambda\|_E^p\,dt,
  \]
  where
  \[
  \|A\ep{w}_\lambda\|_E = \|A(I+\varepsilon A)^{-1}v_\lambda\|_E
  \lesssim_\varepsilon \|v_\lambda\|_E
  \]
  and
  \[
  \|\ep{g}_\lambda\|_E = 
  \| (I+\varepsilon A)^{-1} f_\lambda(v_\lambda+G_A) \|_E
  \leq \| f_\lambda(v_\lambda+G_A) \|_E \lesssim_\lambda
  \|v_\lambda\|_E + \|G_A\|_E,
  \]
  hence
  \[
  \E \Big( \int_0^T \|b(t)\|_E\,dt \Big)^p \lesssim_{\lambda,\varepsilon,T}
  \| v_\lambda \|_{\H_p}^p + \| G_A \|_{\H_p}^p < \infty,
  \]
  which justifies applying It\^o's formula. Setting
  $\psi(x):= \|x\|_E^p$, we have
  \[
  \psi(\ep{w}_\lambda) + \int_0^t
  \ip{A\ep{w}_\lambda+\ep{g}_\lambda}{\psi'(\ep{w}_\lambda)}\,ds =
  \int_0^t \psi'(\ep{w}_\lambda)B(s)\,dW(s) + R(t),
  \]
  where $R$ is a ``remainder'' term, the precise definition of which
  is given in \cite{vNZ}.
  Note that $\psi(u)=\big(\|u\|_E^2\big)^{p/2}$ and $\psi'(u) =
  p|u|^{p-2}u=p \|u\|_E^{p-2}J(u)$, where $J$ is the duality mapping
  of $E$,
  \[
  J: u \mapsto u|u|^{p-2}\|u\|_E^{2-p},
  \]
  i.e. $J$ is the G\^ateaux (and Fr\'echet) derivative of
  $\|\cdot\|_E^2/2$. Since $A$ is $m$-accretive on $E$, it holds
  \[
  \ip{A\ep{w}_\lambda}{\psi'(\ep{w}_\lambda)} =
  p \|\ep{w}_\lambda\|_E^{p-2}
  \ip{A\ep{w}_\lambda}{J(\ep{w}_\lambda)} \geq 0.
  \]
  Moreover, there exists $\delta>0$ and $N=N(\delta)>0$ such that
  (cf. \cite{vNZ})
  \[
  \E\sup_{t \leq T} |R(t)| \leq \delta \E\sup_{t\leq T}
  \|\ep{w}_\lambda(t)\|_E^p + N \E\Big( \int_0^T
  \|B(s)\|^2_{\gamma(H \to E)}\,ds\Big)^{p/2}
  \]
  and, by some calculations based on Young's and Burkholder's
  inequalities,
  \begin{multline*}
  \E\sup_{t\leq T} \Big| \int_0^t \psi'(\ep{w}_\lambda(s))B(s)\,dW(s)
  \Big| \\
  \lesssim \delta \E\sup_{t\leq T} \|\ep{w}_\lambda(t)\|_E^p +
  N \E\Big( \int_0^T \|B(s)\|^2_{\gamma(H \to E)}\,ds\Big)^{p/2}.
  \end{multline*}
  We thus arrive at the estimate
  \[
  \E\sup_{t \leq T} \|\ep{w}_\lambda\|_E^p \lesssim \delta
  \E\sup_{t\leq T} \|\ep{w}_\lambda\|_E^p + \|B\|^p_{\LL_p^\gamma} +
  \E\sup_{t\leq T} \int_0^t 
  \ip{-\ep{g}_\lambda}{\ep{w}_\lambda|\ep{w}_\lambda|^{p-2}}\,ds.
  \]
  Letting $\varepsilon \to 0$, we are left with
  \[
  \E\sup_{t \leq T} \|v_\lambda\|_E^p \lesssim \delta
  \E\sup_{t\leq T} \|v_\lambda\|_E^p + \|B\|^p_{\LL_p^\gamma} +
  \E\sup_{t\leq T} \int_0^t
  \bip{-\tilde{f}_\lambda(s,v_\lambda(s))}{\psi'(v_\lambda(s))}\,ds.
  \]
  Note that we have
  \[
  \ip{\tilde{f}_\lambda(t,v_\lambda)}{\psi'(v_\lambda)} =
  p\|v_\lambda\|_E^{p-2}
  \ip{\tilde{f}_\lambda(t,v_\lambda)}{J(v_\lambda)} =
  p \|v_\lambda\|_E^{p-2}
  \ip{f_\lambda(G_A+v_\lambda)}{J(v_\lambda)},
  \]
  where, by accretivity of $f_\lambda$,
  \begin{align*}
    \ip{f_\lambda(G_A+v_\lambda)}{J(v_\lambda)} &=
    \ip{f_\lambda(G_A+v_\lambda)-f(G_A)}{J(G_A+v_\lambda-G_A)}
    + \ip{f_\lambda(G_A)}{J(v_\lambda)}\\
    &\geq \ip{f_\lambda(G_A)}{J(v_\lambda)},
  \end{align*}
  hence, recalling that $\psi'(u)=p u|u|^{p-2}$,
  \[
  \ip{\tilde{f}_\lambda(t,v_\lambda)}{\psi'(v_\lambda)} \geq p
  \ip{f_\lambda(G_A)}{v_\lambda|v_\lambda|^{p-2}},
  \]
  and, by Young's inequality with conjugate exponents $p$ and
  $p'=p/(p-1)$,
  \[
  \ip{f_\lambda(G_A)}{v_\lambda|v_\lambda|^{p-2}} \lesssim
  N \|f_\lambda(G_A)\|^p_{L_p} + \delta
  \|v_\lambda|v_\lambda|^{p-2}\|^{p'}_{L_{p'}} = N \|f_\lambda(G_A)\|^p_{L_p}
  + \delta \|v_\lambda\|^p_{L_p},
  \]
  so that
  \begin{align*}
    \E\sup_{t\leq T} \Big| \int_0^t
    \ip{\tilde{f}_\lambda(t,v_\lambda)}{\psi'(v_\lambda)}\,ds \Big|
    &\lesssim \delta \E\sup_{t\leq T} \|v_\lambda(t)\|_E^p
    + N \E\sup_{t \leq T} \|f_\lambda(G_A(t))\|_E^p\\
    &\lesssim 1 + \delta \E\sup_{t\leq T} \|v_\lambda(t)\|_E^p +
    N \|G\|_{\mathcal{L}_{p^*}},
  \end{align*}
  where the last constant does not depend on $\lambda$.  

  Combining the above estimates and choosing $\delta$ small enough, we
  are left with
  \[
  \E\sup_{t \leq T} \|v_\lambda\|_E^p \lesssim 1 + \E\|u_0\|_E^p
  + \|G\|^{p^*}_{\LL_{p^*}} + \|B\|^p_{\LL_p^\gamma},
  \]
  with implicit constant independent of $\lambda$.
\end{proof}

Thanks to the a priori estimate just established, we are now going to
show that $\{u_\lambda\}_\lambda$ is a Cauchy sequence in
$\mathbb{H}_2$, hence that there exists $u \in \mathbb{H}_2$ such that
$u_\lambda \to u$ in $\mathbb{H}_2$ as $\lambda \to 0$.  In
particular, we have
\[
d(u_\lambda-u_\mu) + A(u_\lambda-u_\mu)\,dt 
+ (f_\lambda(u_\lambda) - f_\mu(u_\mu))\,dt = 0,
\]
from which we obtain, using the same argument as in
\cite[pp.~1539-1540]{cm:EJP10},
\begin{align*}
\E \sup_{t \leq T} \|u_\lambda-u_\mu\|_{L_2}^2 &\lesssim_T (\lambda+\mu)
\big( \E\sup_{t \leq T} \|f_\lambda(u_\lambda(t))\|_{L_2}^2
+ \E\sup_{t \leq T} \|f_\mu(u_\mu(t))\|_{L_2}^2 \big)\\
&\lesssim (\lambda+\mu) \big( 1 + \E\sup_{t \leq T} \|u_\lambda(t)\|_{L_p}^p
+ \E\sup_{t \leq T} \|u_\mu(t)\|_{L_p}^p \big).
\end{align*}
Since
\[
\| u_\lambda \|_{\H_p} \leq \| v_\lambda \|_{\H_p} + 
\| G_A \|_{\H_p}
\]
and $\|G_A\|_{\H_p}$ is finite because $G \in \LL^m_{p^*}$, we conclude that
$\E\sup_{t \leq T}\|u_\lambda(t)\|_E^p$ is bounded uniformly over
$\lambda$, hence that there exists $u \in \H_2$ such that $u_\lambda
\to u$ in $\H_2$ as $\lambda \to 0$.

As in \cite{cm:EJP10}, one can now pass to the limit as $\lambda \to
0$ in (\ref{eq:regm}), concluding that $u$ is indeed a mild solution
of (\ref{eq:musso}). Since $\E\sup_{t\leq T}\|u\|_{L_p}^p < \infty$,
one also gets that $f(u) \in L_1([0,T] \to H)$, hence, by the
uniqueness results in \cite{cm:IDAQP10}, $u$ is the unique c\`adl\`ag
mild solution belonging to $\H_2$.

\subsection{Proof of Theorem \ref{thm:gen}}
We need the following lemma, whose proof is completely analogous to
the proof of \cite[Lemma~9]{cm:EJP10}, hence omitted.
\begin{lemma}
  Assume that $u_{01}$, $u_{02} \in \L_p$; $B_1$, $B_2 \in
  \LL^\gamma_p$; $G_1$, $G_2 \in \LL^m_{p^*}$, and denote the unique
  c\`adl\`ag mild solutions of
  \[
  du + Au\,dt + f(u)\,dt = B_1\,dW + \int_Z G_1\,d\m, \qquad u(0)=u_{01}
  \]
  and
  \[
  du + Au\,dt + f(u)\,dt = B_2\,dW + \int_Z G_2\,d\m, \qquad u(0)=u_{02}
  \]
  by $u_1$ and $u_2$, respectively. Then one has
  \begin{multline}
    \label{eq:genlip}
    \E\sup_{t \leq T} \|u_1(t) - u_2(t)\|_H^2 \lesssim_T
    \E\|u_{01}-u_{02}\|_H^2 \\
    + \E\int_0^T \|B_1(t)-B_2(t)\|^2_{\gamma(H \to H)}\,dt
    + \E\int_0^T\!\!\int_Z \|G_1(t,z)-G_2(t,z)\|^2_H\,m(dz)\,dt.
  \end{multline}
\end{lemma}

\medskip

Let us consider sequences $\{u_{0n}\}_n \subset \L_p$, $\{B_n\}_n
\subset \LL^\gamma_p$ and $\{G_n\}_n \subset \LL^m_{p^*}$ such that
$u_{0n} \to u_0$ in $\L_2$, $B_n \to B$ in $\LL^\gamma_2$ and $G_n \to
G$ in $\LL^m_2$ as $n \to \infty$. Denoting by $u_n$ the unique mild
solution in $\H_2$ of
\[
du_n + Au_n\,dt + f(u_n)\,dt = B_n\,dW + \int_Z G_n\,d\m,
\qquad u_n(0)=u_{0n},
\]
the previous lemma yields
\begin{multline*}
  \E\sup_{t \leq T} \|u_n(t) - u_m(t)\|_H^2 \lesssim_T
  \E\|u_{0n}-u_{0m}\|_H^2 \\
  + \E\int_0^T \|B_n(t)-B_m(t)\|^2_{\gamma(H \to H)}\,dt +
  \E\int_0^T\!\!\int_Z \|G_n(t,z)-G_m(t,z)\|^2_H\,m(dz)\,dt,
\end{multline*}
i.e. $\{u_n\}_n$ is a Cauchy sequence in $\H_2$. This implies that
$u_n \to u$ in $\H_2$ as $n \to \infty$, and $u$ is a generalized
solution of (\ref{eq:musso}). Since the limit does not depend on the
choice of $u_{0n}$, $B_n$ and $G_n$, the generalized solution is
unique.

\bibliographystyle{amsplain}
\bibliography{ref.bib}

\end{document}